\documentclass[12pt]{amsart}
\usepackage{mathrsfs}
\usepackage{amsfonts}
\usepackage{amsthm}
\usepackage{amssymb}
\usepackage{latexsym}
\newtheorem{theorem}{Theorem}
\newtheorem{lemma}{Lemma}

\newtheorem{proposition}{Proposition}

\begin{document}

\title[The complex Green operator: compactness]{The complex Green operator on CR-submanifolds of $\mathbb{C}^{n}$ of hypersurface type: compactness}

\author{Emil J. Straube}
\address{Department of Mathematics\\
Texas A\&M University\\
College Station, Texas, 77843}
\email{straube@math.tamu.edu}

\thanks{Research supported in part by NSF grant DMS 0758534}

\date{July 5, 2010; revised August 7, 2010.}

\begin{abstract}
We establish compactness estimates for $\overline{\partial}_{b}$ on a compact pseudoconvex CR-submanifold of $\mathbb{C}^{n}$ of hypersurface type that satisfies property(P). 
When the submanifold is orientable, these estimates were proved by A.~Raich in \cite{Raich10} using microlocal methods. Our proof deduces the estimates from (a slight extension, when $q>1$, of) those known on hypersurfaces via the fact that locally, CR-submanifolds of hypersurface type are CR-equivalent to a hypersurface. The relationship between two potential theoretic conditions is also clarified.
\end{abstract}

\maketitle

\section{Introduction}\label{intro}
A CR-submanifold $M$ of $\mathbb{C}^{n}$ is of hypersurface type if the complex tangent bundle has real codimension $1$ inside the real tangent bundle of $M$. For background on CR-manifolds, we refer the reader to the books \cite{BER, Boggess91, ChenShaw01, Zampieri08}. The $\overline{\partial}$-complex of the ambient $\mathbb{C}^{n}$ induces the (extrinsic) $\overline{\partial}_{b}$-complex on $M$ (\cite{Boggess91}, chapter 8, \cite{ChenShaw01}, chapter 7). In this paper, we will study the associated complex Green operator on a compact pseudoconvex CR-submanifold of $\mathbb{C}^{n}$ of hypersurface type. These submanifolds are (possibly non-orientable) generalizations, to codimension greater than one, of boundaries of bounded pseudoconvex domains.\footnote{When orientable, they are boundaries of complex analytic varieties, but in general only in the sense of currents (\cite{HarveyLawson75}). This is not strong enough to apply known methods as in \cite{Kohn86, RaichStraube08}.} 

That $\overline{\partial}_{b}$, hence $\overline{\partial}_{b}^{*}$ and $\Box_{b}$, have closed range in $\mathcal{L}^{2}(M)$ when $M$ is the boundary of a smooth bounded pseudoconvex domain in $\mathbb{C}^{n}$, was shown in \cite{Shaw85b, BoasShaw86, Kohn86}. Compactness estimates for these operators when the boundary of the domain satisfies a potential theoretic condition known as property($P$) were established only recently in \cite{RaichStraube08} (the results in \cite{RaichStraube08}, and ours, are actually more precise, see section \ref{comp} below; see also \cite{Kohn02}). For general $M$, the closed range properties are also surprisingly recent: they were established in \cite{Nicoara06} for $M$ with $\dim(M) \geq 5$ (compare also \cite{HarRaich10}); the case $\dim(M)=3$ is open. Subsequently, Raich further developed the microlocal methods of \cite{Nicoara06} to derive compactness estimates when $M$ satisfies a CR-analogue of property(P) and is also orientable (\cite{Raich10}). Because compactness is a local property, one would not expect orientability to be crucial. In this paper we give a new proof of the compactness results in \cite{Raich10} that indeed does not require orientability. More specifically, because locally, $M$ is CR-equivalent to a hypersurface, these results should follow directly from the corresponding results on hypersurfaces in our joint work in \cite{RaichStraube08}. \emph{Prima facie} this seems straightforward (again because compactness is a local property). However, two issues arise immediately. First, $\overline{\partial}_{b}$ does not commute with pullbacks under CR-mappings. This turns out to be easy to handle (in the context of compactness).\footnote{Using a metric free version of the $\overline{\partial}_{b}$-complex (see e.g. \cite{Boggess91}) would also remedy this situation. But the results in \cite{RaichStraube08} are for the extrinsic complex, so an argument has to be made somewhere. Here, the point is simply that the behavior of the extrinsically defined $\overline{\partial}_{b}$ under pullbacks notwithstanding, compactness is preserved under CR-equivalences, see Remark 6 at the end of section \ref{proof}.} Second, and more importantly, the behavior under biholomorphisms and under CR-equivalences of the version of property($P$) appropriate for $(0,q)$-forms is not understood when $q>1$. This is of independent interest also for $\overline{\partial}$, because compactness of $\overline{\partial}$ is known to be invariant. The same remark applies to the behavior under a change in the metric (see \cite{CelikStraube08} for the invariance of compactness). We sidestep the question of invariance by showing that `property($P$) with respect to a general metric' still implies the desired compactness estimates.

We also clarify the relationship between seemingly different versions of `property(P)-type' conditions. In particular, we show that for submanifolds of $\mathbb{C}^{n}$ as above, property(P) and its CR-analogue from \cite{Raich10} are actually equivalent.

Compactness estimates for $\overline{\partial}_{b}$ imply that the complex Green operator is continuous in Sobolev norms (see the discussion at the end of section \ref{comp}). But it is also of interest to study estimates in Sobolev norms in situations where compactness may not hold. When $M$ is the boundary of a smooth bounded pseudoconvex domain in $\mathbb{C}^{n}$, such estimates were established in \cite{BoasStraube91} when the domain admits a defining function that is plurisubharmonic at points of the boundary (i.e. $M$). This class of domains includes the convex domains. We plan to address the question of Sobolev estimates in the case of submanifolds of hypersurface type in a future paper.

When $M$ is not assumed of hypersurface type, it is not known whether the results for $\overline{\partial}_{b}$ discussed above hold; it is not obvious that the methods of \cite{Nicoara06, RaichStraube08, Raich10} carry over. What the methods of the present paper show is that, at least for the implication from property(P) to compactness, one can restrict attention to generic submanifolds (if $M$ is of hypersurface type, this means to hypersurfaces). 

Finally, when embedability is dropped from the assumptions, $\overline{\partial}_{b}$ need not have closed range in $\mathcal{L}^{2}$; indeed, the closed range property is quite intimately related to embedability  (see for example \cite{Kohn90}, \cite{Boggess91}, chapters 11 and 12, and \cite{ChenShaw01}, chapter 12, for a discussion of these issues). Whether a compactness estimate (which implies closed range, see the discussion in the next section) has any implications for embedability seems to be unknown.\footnote{However, compactness should yield global regularity properties also in this situation. These properties are favorable for embedability, in view of Burns' modification \cite{Burns79} of the Boutet de Monvel embedding theorem \cite{BdM75}; see Theorem 12.2.2 in \cite{ChenShaw01}. What is not clear is whether points can still be separated when strict pseudoconvexity is replaced by compactness estimates for $\overline{\partial}_{b}$.} Positive results in this direction would certainly provide ample motivation for studying sufficient conditions for compactness also on abstract CR-manifolds.

\section{CR submanifolds and $\overline{\partial}_{b}$}\label{CR}
Let $M \subset \mathbb{C}^{n}$ be a compact smooth CR-submanifold. For $z \in M$, the complex tangent space $T^{\mathbb{C}}_{z}(M)$ is defined as $T_{z}(M) \cap JT_{z}(M)$, where $J$ is multiplication by $i$. The spaces $T^{\mathbb{C}}_{z}(M)$ form the complex tangent bundle $T^{\mathbb{C}}(M)$. $T^{1,0}(M)$ denotes the bundle of complex vector fields of type $(1,0)$ tangent to $M$. $T^{1,0}_{z}(M)$ and $T^{\mathbb{C}}_{z}(M)$ are $\mathbb{C}$-linearly isomorphic via $\sum_{j=1}^{n}a_{j}\partial/\partial z_{j} \leftrightarrow (a_{1}, \cdots, a_{n})$. $M$ is said to be of hypersurface type, if at each point $z \in M$, $T^{\mathbb{C}}_{z}(M)$ has (real) codimension one inside $T_{z}(M)$. That is, if $m-1$ is the complex dimension of $T^{\mathbb{C}}_{z}(M)$ for $z \in M$ (the CR-dimension of $M$), then the real dimension of $M$ is $2m-1$. (The notation $m-1$ is chosen here in analogy to a hypersurface in $\mathbb{C}^{m}$, and will be convenient later.) Locally, near a point $z \in M$, chose a real vector field $T \in T(M)$ that is transversal to $T^{\mathbb{C}}(M)$ at each point. 'The' Levi form at $z$ is the Hermitian form $L_{z}$ on $T^{1,0}_{z}(M)$ given by
\begin{equation}\label{levi}
\frac{1}{2i}[X,\overline{Y}] = L_{z}(X,\overline{Y})\,T \mod \;\;T^{1,0}_{z}(M) \oplus T^{0,1}_{z}(M) ,\;\;\;\; X,Y \in T^{1,0}_{z}(M) \; .
\end{equation}
$M$ is pseudoconvex if every point has a neighborhood where the Levi form is (positive or negative) semidefinite (the latter property is independent of the choice of $T$). In the case of an embedded hypersurface, this definition agrees with the familiar one. Note that we are not assuming $M$ to be orientable, so such $T$'s need not be defined globally. 

$\Lambda^{0,q}(M)$ denotes as usual the bundle of $(0,q)$-forms on $M$, that is, the bundle obtained from $\Lambda^{0,q}(\mathbb{C}^{n})$ by first restricting (the bundle) to $M$ and then quotenting out, at each point, the (restriction to the point of) the ideal generated by a set of local defining functions and their $\overline{\partial}$'s (see \cite{Boggess91}, chapter 8, \cite{ChenShaw01}, chapter 7). Alternatively, this is the bundle of skew symmetric multilinear maps on $(T^{0,1}(M))^{q}$ to $\mathbb{C}$. Locally, these forms can be described as follows. Let $L_{1}, \cdots, L_{m-1}$ be an orthonormal basis (locally) of $T^{1,0}(M)$, where the inner product on vectors is the one inherited from the ambient $\mathbb{C}^{n}$. Denote by $\omega_{1}, \cdots, \omega_{m-1}$ the dual basis in $(T^{1,0}(M))^{*} = \Lambda^{1,0}(M)$. Then a $(0,q)$-form $u$ can be written uniquely as $u = \sum^{\prime}_{|J|=q}u_{J}\overline{\omega}^{J}$, where $\prime$ indicates summation over strictly increasing multi-indices $J=(j_{1}, \cdots, j_{q}) \in \{1, \cdots, m-1\}^{q}$ only, and $\overline{\omega}^{J} = \overline{\omega_{j_{1}}}\wedge \cdots \wedge \overline{\omega_{j_{q}}}$. We still may conveniently take the coefficients $u_{J}$ to be defined for all $J$ by skew symmetry.

The $\overline{\partial}_{b}$-complex on $M$ is induced by the $\overline{\partial}$-complex on $\mathbb{C}^{n}$: if $u$ is a (smooth) section of $\Lambda^{0,q}(M)$, choose a section $\tilde{u}$ of $\Lambda^{0,q}(\mathbb{C}^{n})$ whose tangential part agrees with $u$ at points of $M$. Then $\overline{\partial}_{b}u$ equals the tangential part of $\overline{\partial}\tilde{u}$. This definition is independent of the choice of the extension $\tilde{u}$ (see again \cite{Boggess91, ChenShaw01} for details). In local coordinates as above, this gives that for a function $f$, $\overline{\partial}_{b}f = \sum_{j=1}^{m-1}\overline{L_{j}}(f)\overline{\omega_{j}}$. For a $(0,q)$-from $u = \sum^{\prime}_{|J|=q}u_{J}\overline{\omega}^{J}$ with $q >0$, one has
\begin{equation}\label{d-bar-b}
\overline{\partial}_{b}u = \sideset{}{'}\sum_{|J|=q} \sum_{j=1}^{m-1}\overline{L_{j}}(u_{J})\overline{\omega_{j}}\wedge\overline{\omega}^{J} + \sideset{}{'}\sum_{|J|=q}u_{J}\,\overline{\partial}_{b}\overline{\omega}^{J} \; .
\end{equation}
Note that in the second sum on the right hand side of \eqref{d-bar-b}, the coefficients $u_{J}$ are not differentiated.

The inner product in $\mathbb{C}^{n}$ induces a pointwise inner product in $T^{1,0}(M)$ and therefore in $\Lambda^{0,1}(M)$ (by duality). Declaring $\overline{\omega}^{J}$ and $\overline{\omega}^{I}$ orthonormal when $J \neq I$ (as sets) results in a pointwise inner product on $\Lambda^{0,q}(M)$. (Note that this inner product is well defined: it is independent of the choice of local orthonormal basis $L_{1}, \cdots, L_{m-1}$ as above.) This inner product, denoted by $(\;,\;)_{z}$, then induces an $\mathcal{L}^{2}$ inner product via integration against Lebesgue measure $d\mu_{M}$ on $M$ (induced from the ambient $\mathbb{C}^{n}$):
\begin{equation}\label{L2}
(u,v)_{M} = \int_{M} (u,v)_{z}d\mu_{M}(z) \; .
\end{equation}
We denote by $\mathcal{L}^{2}_{(0,q)}(M)$ the Hilbert space of $\mathcal{L}^{2}$-sections of $\Lambda^{0,q}(M)$, provided with the inner product \eqref{L2}. Locally, elements of $\mathcal{L}^{2}_{(0,q)}(M)$ are still written as above, but the coefficients $u_{J}$ are now in $\mathcal{L}^{2}$ (locally). $\overline{\partial}_{b}$ is defined from $\mathcal{L}^{2}_{(0,q)}(M)$ to $\mathcal{L}^{2}_{(0,q+1)}(M)$ as a closed, densely defined operator in the usual manner: a form $u$ is in the domain of $\overline{\partial}_{b}$ if its coefficients, computed in local coordinates in the sense of distributions via \eqref{d-bar-b}, are square integrable. The resulting complex is the $\mathcal{L}^{2}$-$\overline{\partial}_{b}$-complex. 

As a closed, densely defined operator, $\overline{\partial}_{b}$ has an adjoint $\overline{\partial}_{b}^{*}:\mathcal{L}^{2}_{(0,q+1)}(M) \rightarrow \mathcal{L}^{2}_{(0,q)}(M)$. Integration by parts gives that in local coordinates, if u = $\sum_{|J|=q}^{\prime}u_{J}\overline{\omega}^{J}$, then
\begin{equation}\label{d-bar-b-*}
\overline{\partial}_{b}^{*}u = - \sum_{j=1}^{m-1}\sideset{}{'}\sum_{|K|=q-1}L_{j}u_{jK}\overline{\omega}^{K} + \text{terms of order zero}  \end{equation}
(compare \cite{FollandKohn72}, p.94, \cite{ChenShaw01}, section 8.3). Finally, $\Box_{b,q}$ is defined by
$\Box_{b,q}u = \overline{\partial}_{b}\overline{\partial}_{b}^{*}u + \overline{\partial}_{b}^{*}\overline{\partial}_{b}u$, with domain coinsisting of those forms in $\mathcal{L}^{2}_{(0,q)}(M)$ where the compositions are defined. $\Box_{b,q}$ is selfadjoint; in fact, it is the unique selfadjoint operator associated with the closed quadratic form (see \cite{Davies95}, Theorem 4.4.2, \cite{ReedSimon80}, Theorem VIII.15) $Q_{b,q}(u,u)=(\overline{\partial}_{b}u,\overline{\partial}_{b}u)_{M} + (\overline{\partial}_{b}^{*}u,\overline{\partial}_{b}^{*}u)_{M}$, with form domain equal to $dom(\overline{\partial}_{b}) \cap dom(\overline{\partial}_{b}^{*})$. The proof of this equality is essentially the same as that of Proposition 2.8 in \cite{Straube10}, where the corresponding fact in the context of the $\overline{\partial}$-complex is proved.

\section{Compactness of the complex Green operator}\label{comp}
Recall that a compact set $K \subset \mathbb{C}^{n}$ is said to satisfy property($P_{q}$) if the following holds: for every $A>0$ there is a $C^{2}$ function $\lambda_{A}$ defined in some neighborhood $U_{A}$ of $K$ such that 
\begin{equation}\label{Pq1}
0 \leq \lambda_{A}(z) \leq 1 \; , \; z \in U_{A} \; ,
\end{equation}
and
\begin{equation}\label{Pq2}
\sideset{}{'}\sum_{|K|=q-1}\sum_{j,k=1}^{n}\frac{\partial^{2}\lambda_{A}(z)}{\partial z_{j}\partial\overline{z_{k}}}w_{jK}\overline{w_{kK}} \geq A|w|^{2} \; , \; z \in U_{A} \;, \; w \in \Lambda_{z}^{0,q} \; ,
\end{equation}
where $\Lambda_{z}^{0,q}$ denotes the space of $(0,q)$-forms at $z$. For $q=1$, and $K$ the boundary of a pseudoconvex domain, this property was introduced in \cite{Catlin84b}; Catlin showed that when the boundary is sufficiently smooth (this restriction was removed in \cite{Straube97}) and satisfies the condition, then the $\overline{\partial}$-Neumann operator on $(0,1)$-forms on the domain is compact. In the above form, still for $q=1$, the condition was studied in detail (under the name $B$-regularity) in \cite{Sibony87b}. For $q>1$, the property appears in \cite{FuStraube98, FuStraube99}. Sibony's analysis carries over to the $q>1$ case in a straightforward manner. Details, including the relevance of property($P_{q}$) for compactness of the $\overline{\partial}$-Neumann operator on $(0,q)$-forms, may be found in \cite{Straube10}, chapter 4. In particular, Lemma 4.7 there gives two useful equivalent formulations of \eqref{Pq2}. First, the sum of any $q$ (equivalently, the smallest $q$) eigenvalues of the complex Hessian of $\lambda_{A}$ at $z$ is at least $A$. It is immediate from this formulation that property($P_{q}$) implies property($P_{q+1}$). Second,
\begin{equation}\label{Pq2a}
\sum_{s=1}^{q}\sum_{j,k=1}^{n}\frac{\partial^{2}\lambda_{A}(z)}{\partial z_{j}\partial\overline{z_{k}}}(\underline{t}^{s})_{j}\overline{(\underline{t}^{s})_{k}} \,\geq \, A \; ,
\end{equation}
whenever $\underline{t}^{1}, \cdots, \underline{t}^{q}$ are orthonormal in $\mathbb{C}^{n}$.

We can now formulate the main result of this paper, which says that property($P_{q}$) is also sufficient for compactness in the $\overline{\partial}_{b}$-complex at the level of $(0,q)$-forms for compact pseudoconvex CR-submanifolds of hypersurface type, modulo a modification that is necessary. Namely, the property has to be assumed at the symmetric form levels $q$ and $m-1-q$, where where $m-1$ is the CR-dimension of $M$.
\begin{theorem}\label{main}
Let $M \subset \mathbb{C}^{n}$ be a smooth compact pseudoconvex CR-submanifold of hypersurface type. Assume that $M$ satisfies property($P_{q}$) and property($P_{m-1-q}$) (equivalently: $(P_{min\{q,m-1-q\}})$), where $1 \leq q \leq m-2$. Then the following compactness estimate holds: for all $\varepsilon > 0$, there exists a constant $C_{\varepsilon}$, such that
\begin{equation}\label{compest}
\|u\|_{\mathcal{L}^{2}_{(0,q)}(M)} \leq \varepsilon \left (\|\overline{\partial}_{b}u\|_{\mathcal{L}^{2}_{(0,q+1)}(M)}  + \|\overline{\partial}_{b}^{*}u\|_{\mathcal{L}^{2}_{(0,q-1)}(M)} \right ) + C_{\varepsilon}\|u\|_{W^{-1}_{(0,q)}(M)} \; ,
\end{equation}
for all $u \in dom(\overline{\partial}_{b}) \cap dom(\overline{\partial}_{b}^{*})$. (By symmetry, the same family of estimates holds for $q$ replaced by $m-1-q$.) 
\end{theorem}
Note that the condition $1 \leq q \leq m-2$ implies $m \geq 3$, so that the (real) dimension of $M$ is at least five. Formally, the restriction (which is also present in \cite{RaichStraube08, Raich10}) arises because if $q=0$ or $q=m-1$, the assumption would involve property($P_{0}$), and it is not clear what an appropriate interpretation of $P_{0}$ is. When $m$ is at least three, one can obtain a 'substitute' to the effect that ($P_{1}$) is a sufficient condition for compactness of the complex Green operator on $(0,0)$-forms (functions) and on $(0,m-1)$-forms. Details are in \cite{RaichStraube08}, last paragraph of section 1, and \cite{Raich10}, discussion following the statement of Theorem 1.1. The case $m=2$, i.e. $dim(M)=3$, is not understood.

That the assumptions in Theorem \ref{main} have to be made at symmetric form levels is essentially dictated by Lemma \ref{symm} below, and corresponds to the symmetry between $q$ and $(n-1-q)$-forms in the case of the boundary of a domain treated in \cite{RaichStraube08}. In turn, the phenomenon that a condition at the level of $q$-forms in the interior of a domain needs to be imposed at symmetric levels on the boundary to obtain the analogous estimates goes back at least to the conditions $Z(q)$ and $Y(q)$ (which equals $Z(q)$ \emph{and} $Z(n-1-q)$), the conditions for subelliptic $1/2$-estimates in the interior and on the boundary, respectively (compare for example \cite{FollandKohn72}, p.57,94, and \cite{ChenShaw01}, p.192--193). 
\begin{lemma}\label{symm}
Let $M \subset \mathbb{C}^{n}$ be a smooth compact pseudoconvex CR-submanifold of hypersurface type, $1 \leq q \leq m-1$, where $(m-1)$ is the CR-dimension of $M$. A compactness estimate \eqref{compest} holds for $(0,q)$-forms if and only if it holds for $(0,m-1-q)$-forms.
\end{lemma}
\begin{proof} For the case of subelliptic estimates, Lemma \ref{symm} is contained in \cite{Koenig04}, Corollary 6.7 and its proof. The proof there relies on a Hodge-$*$ type construction that results in isomorphisms between $\mathcal{L}^{2}_{(0,q)}(M)$ and $\mathcal{L}^{2}_{(0,m-1-q)}(M)$ which intertwine $\overline{\partial}_{b}$ and $\overline{\partial}_{b}^{*}$, modulo terms of order zero (compare also \cite{Kohn81}, page 255; an alternative construction based on a natural sesquilinear pairing is in the appendix of \cite{RaichStraube08}). These terms can be absorbed. Koenig's proof works verbatim for Lemma \ref{symm}.
\end{proof}

\emph{Remark 1}: The reason why the constructions in the proof of Lemma \ref{symm} work is that $M$ is a manifold without boundary, so that being in the domain of $\overline{\partial}_{b}^{*}$ does not involve a boundary condition. The analogous constructions on a domain fail because they result in forms which need not be in the domain of $\overline{\partial}^{*}$ (and the symmetry between form levels with respect to subellipticity and compactness is indeed absent in general on a domain in $\mathbb{C}^{n}$). In addition to the references already given, the reader should also consult \cite{Kohn02}, \cite{HarRaich10}, and \cite{Khanh10} in this connection.

\vspace{0.1in}
\emph{Remark 2}: There is a version of property($P_{q}$) introduced in \cite{McNeal02}, called property($\widetilde{P}_{q}$) there and shown to imply compactness of $N_{q}$. ($P_{q}$) implies ($\widetilde{P}_{q}$), but to what extent the latter is actually more general is not understood at present. For this reason, and because the proof of Theorem \ref{main} requires some of the analysis of property($P_{q}$) from \cite{Sibony87b, FuStraube99, Straube10} that is not (yet) available for ($\widetilde{P}_{q}$), we have chosen to work with ($P_{q}$). 

There are also geometric sufficient conditions for compactness of the $\overline{\partial}$-Neumann operator in terms of short time flows generated by suitable complex tangential vector fields. These ideas, whose relationship to $P$/$\widetilde{P}$ is not understood in general, were introduced for the case of $\mathbb{C}^{2}$ in \cite{Straube04} and generalized for $\mathbb{C}^{n}$ in \cite{MunasingheStraube06} (see also \cite{Straube10}, section 4.11). The geometric arguments from these papers are expected to apply in the context of the present paper as well, but we do not pursue this direction here.

\vspace{0.15in}
Theorem \ref{main} is proved in section \ref{proof}. We conclude this section by noting that the compactness estimate \eqref{compest} implies first the 'usual' $\mathcal{L}^{2}$-theory, and then the 'customary' Sobolev estimates. In particular, $H_{q}:=ker(\Box_{b,q})$ equals $ker(\overline{\partial}_{b}) \cap ker(\overline{\partial}_{b}^{*})$ and is finite dimensional, both $\overline{\partial}_{b}$ and $\overline{\partial}_{b}^{*}$ have closed range, and we have the estimate 
\begin{multline}\label{L2-estimate}
\|u\|_{\mathcal{L}^{2}_{(0,q)}(M)} \leq C \left (\|\overline{\partial}_{b}u\|_{\mathcal{L}^{2}_{(0,q+1)}(M)}  + \|\overline{\partial}_{b}^{*}u\|_{\mathcal{L}^{2}_{(0,q-1)}(M)} \right )\;, \\ 
\; u \in dom(\overline{\partial}_{b}) \cap dom(\overline{\partial}_{b}^{*}) \;, \; u \perp H_{q} \; .
\end{multline}
(see for example \cite{Hormander65}, Theorems 1.1.3 and 1.1.2). It follows further that $\Box_{b,q}$ maps the orthogonal complement of $H_{q}$ onto itself, and  has a bounded inverse there, the complex Green operator $G_{q}$ (compare for example the proof of Theorem 9.4.2 in \cite{ChenShaw01}). We take $G_{q}$ to be defined on all of $\mathcal{L}^{2}_{(0,q)}(M)$, extending it linearly so that it is zero on $H_{q}$. This amounts to precomposing $G_{q}$ with the orthogonal projection onto $(H_{q})^{\perp}$, and so does not affect compactness of the operator. Finally, \eqref{compest} also shows, for $1 \leq q \leq n-2$, that $G_{q}$ is compact (see for example \cite{Straube10}, Proposition 4.2, were these arguments are detailed in the context of the $\overline{\partial}$-Neumann operator.) Once the $\mathcal{L}^{2}$-theory is established, the various estimates can be lifted by standard methods to Sobolev spaces (again using \eqref{compest}) to obtain that the ranges in question are closed also in $W^{s}(M)$ for $s \geq 0$, and that $G_{q}$ is compact (in particular continuous) in $W^{s}_{(0,q)}(M)$, $s \geq 0$ (see for example \cite{Raich10}, subsection 5.3; \cite{KohnNirenberg65}, (proof of) Theorem 3). 

\section{Proof of Theorem \ref{main}}\label{proof}
We begin with some geometric preparations. Fix $P \in M$. Near $P$, $M$ is a graph over a smooth hypersurface in $\mathbb{C}^{m}$ (see \cite{BER, Boggess91}). After suitably rotating coordinates, we may assume that the graphing function is just the inverse of the projection $\pi: \mathbb{C}^{n} \rightarrow \mathbb{C}^{m}$, $(z_{1}, \cdots, z_{n}) \longmapsto (z_{1}, \cdots, z_{m})$, and $\pi(P) = 0$:
\begin{multline}\label{graph}
(z_{1}, \cdots, z_{m}) \longmapsto (z_{1}, \cdots, z_{m}, h_{1}(z_{1}, \cdots, z_{m}), \cdots, h_{n-m}(z_{1}, \cdots, z_{m}) \; , \\
(z_{1}, \cdots, z_{m}) \in \pi(M) \; ,\;\;\;\;\;\;\;
\end{multline}
where $h_{j}$, $1 \leq j \leq n-m$ is a CR-function on $\pi(M)$, and the map \eqref{graph} is a CR-diffeomorphism, namely the inverse of the restriction of $\pi$ to $M$. Because $M$ is pseudoconvex, $\pi(M)$ is a pseudoconvex hypersurface in $\mathbb{C}^{m}$. Denote by $\Omega$ a smooth pseudoconvex domain whose boundary contains a neighborhood of $0$ in $\pi(M)$, and whose other boundary points are strictly pseudoconvex. This can be done along the lines of the construction in section 4 of \cite{Bell86}, ignoring the parts that refer to finite type (compare also \cite{Amar84, GaySebbar85}). These constructions start with a domain, rather than just with a piece of hypersurface (such as $\pi(M)$), but everything is local, including a defining function such that a suitable power is strictly plurisubharmonic, on the side of $\pi(M)$ where it is negative. In principle, $\pi(M)$ could be pseudoconvex from both sides, that is, Levi-flat, in which case we would just choose one side as the pseudoconvex side. However,this does not happen in our situation. By assumption, $M$ satisfies property($P_{q}$), hence property($P_{m-1}$). Therefore, it cannot contain complex varieties of dimension $(m-1)$ (\cite{Straube10}, Lemma 4.20; the statement there is for the boundary of a domain, but the same proof works in the present situation). But then $\pi(M)$ cannot contain such varieties either, and \emph{a fortiori} it cannot be Levi flat.

Near $P$ choose an orthonormal basis $L_{1}, \cdots, L_{m-1}$ for $T^{1,0}(M)$, with associated dual basis $\omega_{1}, \cdots, \omega_{m-1}$, as in section \ref{CR}. Denote by $Z_{1}, \cdots, Z_{m-1}$ the projection of $L_{1}, \cdots, L_{m-1}$; this gives a (local) basis for $T^{1,0}(\pi(M))$. Let $\alpha_{1}, \cdots, \alpha_{m-1}$ be the dual basis. Note that while $Z_{j}=\pi_{*}(L_{j})$, $1 \leq j \leq m-1$, $\alpha_{j} \neq \pi_{*}(\omega_{j})$ in general (compare the discussion in \cite{Boggess91}, section 9.2). For a $(0,q)$-form $u=\sum^{'}_{|J|=q}u_{J}\,\overline{\omega}^{J}$ on $M$ that is supported sufficiently close to $P$, we define a $(0,q)$-from $\hat{\pi}u$ on $\pi(M)$ as follows:
\begin{equation}\label{push}
\hat{\pi}u(z) = \sideset{}{'}\sum_{|J|=q}u_{J}(\pi^{-1}(z))\;\overline{\alpha}^{J} \;, \;  z \in \pi(M) \; .
\end{equation}
Again, note that $\hat{\pi}u \neq \pi_{*}u$ in general; in fact,  $\pi_{*}$ need not take $\Lambda^{0,1}(M) \oplus \Lambda^{0,2}(M) \oplus \cdots \oplus \Lambda^{0,m-1}(M)$ to $\Lambda^{0,1}(\pi(M)) \oplus \cdots \oplus \Lambda^{0,m-1}(\pi(M))$ (see again \cite{Boggess91}, section 9.2, in particular Lemma 2). $\hat{\pi}$ commutes with wedge products. Most importantly for us, it also commutes with $\overline{\partial}_{b}$ and $\overline{\partial}_{b}^{*}$, modulo error terms of order zero. Indeed, by \eqref{d-bar-b}, we have
\begin{multline}\nonumber
\overline{\partial}_{b,\pi(M)}\left (\hat{\pi}u \right )(z) = \sideset{}{'}\sum_{j,J}\overline{Z_{j}}\left (u_{J}(\pi^{-1}(z)\right )\, \alpha_{j}\wedge \overline{\alpha}^{J} + \text{terms of order zero} \\
\stackrel{over}{\longrightarrow} \;\;\;\;\;
\end{multline}
\begin{multline}\label{commute}
\;\;\;\;\;\;\;\;\;\;\;\;\;\;\;= \sideset{}{'}\sum_{j,J}\left (\overline{L_{j}}u_{J} \right )(\pi^{-1}(z))\,\alpha_{j}\wedge \overline{\alpha}^{J} + \text{terms of order zero} \\
\;\;\;\;\;\;\;\;\;\;\;\;\;\;\;\;\;\;\;\;\;\;\;\;= \hat{\pi} \left (\sideset{}{'}\sum_{j,J}\overline{L_{j}}u_{J}\;\omega_{j}\wedge \overline{\omega}^{J} \right ) + \text{terms of order zero} \\
= \hat{\pi}\left (\overline{\partial}_{b,M}u \right )(z) + \text{terms of order zero} \; ,\;\;\;\;\;
\end{multline}
where the additional subscripts on $\overline{\partial}_{b}$ have their obvious meaning, and `terms of order zero' refers to the coefficients $u_{J}(\pi^{-1}(z))$. An analogous computation using \eqref{d-bar-b-*} gives
\begin{equation}\label{commute*}
\overline{\partial}^{*}_{b,\pi(M)}\left (\hat{\pi}u \right )(z) = \hat{\pi}\left (\overline{\partial}^{*}_{b,M}u \right )(z) + \text{terms of order zero} \; .
\end{equation}

To prove the compactness estimate \eqref{compest}, it suffices, via a partition of unity, to prove the estimate for forms that are supported in (small) neighborhoods where the above construction is valid. For such a form $u$, \eqref{push} and \eqref{commute*} show that \eqref{compest} holds if (and only if) a corresponding estimate holds for $\hat{\pi}u$ on $\pi(M)$. This will follow if there is a compactness estimate for $(0,q)$-forms on $b\Omega$ (since $\hat{\pi}u$ is supported near $0$, it can be viewed as a form on $b\Omega$). We are now in a position to give the (simple) proof when $q=1$.

So assume $q=1$. We claim that $b\Omega$ satisfies property($P_{1}$). First, compact subsets of $b\Omega \setminus \pi(M)$ satisfy property($P_{1}$) as compact subsets of a strictly pseudoconvex hypersurface (this is a special case of Proposition 1.12, \cite{Sibony87b}, Proposition 4.15, \cite{Straube10}; it is also implicit in \cite{Catlin84b}). Second, $\pi^{-1}$ can be extended as a $C^{\infty}$-map $F=(F_{1}, \cdots, F_{n})$ into a neighborhood of $\pi(M)$ in such a way that $\overline{\partial}F_{j}$, $1 \leq j \leq n$, vanishes to infinite order on $\pi(M)$ (see for example \cite{Boggess91}, Theorem 2, section 9.1). A standard computation shows that the pullbacks under $F$ of the functions $\lambda_{A}$ with large Hessians (from the definition of property($P_{1}$) for $M$) have large Hessians at points of $\pi(M)$ in directions in the complex tanget space. Adding $B\rho^{2}$, where $\rho$ is a defining function for $\pi(M)$ and $B$ is large enough, produces a function with large Hessian in all directions (compare the arguments in the proof of Proposition \ref{equivalent} in section \ref{potential} below). By continuity, there is a neighborhood of the compact set $b\Omega \cap \pi(M)$, depending on $A$, where the Hessian is at least $A/2$. Also, in a small enough neighborhood, again depending on $A$, the pullbacks will be bounded independently of $A$ by, say, $3/2$. In other words, $b\Omega \cap \pi(M)$ satisfies property($P_{1}$). Therefore, $b\Omega$ is a countable union of compact sets, all of which satisfy property($P_{1}$). Hence so does $b\Omega$ (\cite{Sibony87b}, Proposition 1.9, \cite{Straube10}, Corollary 4.14). Then $b\Omega$ also satisfies ($P_{m-1-q}$) (since $1 \leq m-1-q$), and Theorem 1.4 in \cite{RaichStraube08} gives a compactness estimate for $(0,1)$-forms on $b\Omega$. This completes the proof of Theorem \ref{main} when $q=1$.

When $q>1$, the argument involving the pullbacks of the functions $\lambda_{A}$ no longer works in such a straightforward way. This is essentially the statement that property($P_{q}$) is no longer obviously invariant under biholomorphisms when $q>1$ (the author does not know whether it actually is). The reason for this failure is simple: biholomorphic maps are in general not conformal, so that the push forward of a set of orthonormal vectors as in \eqref{Pq2a} is not orthonormal in general.\footnote{See Remarks 3 and 4 below for further discussion.} We remedy this situation by modifying the Euclidean metric near $\pi(M)$ (see \eqref{inner} below) so that the matrix of the Hessian of the pullback of $\lambda_{A}$, with respect to this modified metric, satisfies \eqref{Pq2a} (see \eqref{hessiancomp} below). We then derive analogues of the weighted estimates (19) in \cite{RaichStraube08} for this metric, both on $\Omega$ and on $\mathbb{B}\setminus \Omega$, where $\mathbb{B}$ is a big ball containing $\overline{\Omega}$ (see \eqref{estimate1} below). Once these estimates are established, the arguments in \cite{RaichStraube08} apply again to first give compactness of the respective $\overline{\partial}$-Neumann operators and then the desired compactness estimate for $(0,q)$-forms on $b\Omega$. 

We start with $\Omega$. It is convenient (although not necessary) to use that in our situation, more can be said about extending CR-functions from $\pi(M)$. Because $\pi(M)$ contains no analytic varieties of dimension $(m-1)$ (as observed at the end of the first paragraph of this section), CR-functions on $\pi(M)$ extend holomorphically to the pseudoconvex side of $\pi(M)$ (\cite{Trepreau86}). Thus locally, $M$ actually is the boundary, in the $C^{\infty}$-sense, of a complex manifold $\widetilde{M}$.\footnote{If $M$ is also orientable, this Kneser-Lewy type extension property combines with \cite{HarveyLawson75} to give that in fact $M$ is globally the boundary of a complex manifold (\cite{Kohn86}, p. 526).} $\widetilde{M}$ is graphed over a one sided neighborhood of $\pi(M)$, the graphing functions $H_{1}(z_{1}, \cdots,z_{m}), \cdots, H_{n-m}(z_{1}, \cdots, z_{m})$ are simply the (one sided) holomorphic extensions of $h_{1}, \cdots, h_{n-m}$ from \eqref{graph}. Extend $L_{1}, \cdots, L_{m-1}$ and augment the set by $L_{m}$ so that these vector fields form an orthonormal basis for $T^{1,0}(\widetilde{M})$ (locally), smooth up to the boundary $M$. Denote by $\omega_{1}, \cdots, \omega_{m}$ the dual basis. Via $\pi_{*}$, this gives fields $Z_{1}, \cdots, Z_{m}$ defined in a one sided neighborhood of $\pi(M)$ (and smooth up to the boundary $\pi(M)$); let $\alpha_{1}, \cdots, \alpha_{m}$ be the dual basis. Next, choose a cutoff function $\chi_{1}$ that equals one near $\pi(M) \cap b\Omega$, that is supported in a sufficiently small neighborhood of $\pi(M) \cap b\Omega$, and that takes values in $[0,1]$.  For a $(0,q)$-from $v$ on $\Omega$, denote by $v_{J}$ the coefficients with respect to $\alpha_{1}, \cdots, \alpha_{m}$; $v_{J}$ is defined on the support of $\chi_{1}$ (shrinking that support if necessary). Recall that $(v_{1},v_{2})_{z}$ denotes the Euclidean inner product between $(0,q)$-forms at the point $z \in \mathbb{C}^{m}$. For $z \in \overline{\Omega}$, we define the following new pointwise inner product on $(0,q)$-forms:
\begin{equation}\label{inner}
\{v_{1}, v_{2}\}_{z} = \chi_{1}(z)\sideset{}{'}\sum_{|J|=q}(v_{1})_{J}\overline{(v_{2})_{J}} + \left (1 - \chi_{1}(z)\right )(v_{1}, v_{2})_{z} \; .
\end{equation}
The norms induced by these inner products are equivalent to the (pointwise) Euclidean norms, uniformly in $z \in \overline{\Omega}$. We introduce weighted $\mathcal{L}^{2}$ inner products
\begin{equation}\label{inner2}
(v_{1}, v_{2})_{\varphi} = \int_{\Omega} \{v_{1}, v_{2}\}_{z}\, e^{-\varphi(z)}dV(z) 
\; ;
\end{equation}
where for the moment, $\varphi$ is only assumed smooth on $\overline{\Omega}$. The reader should keep in mind that these are not just weighted inner products as the term is typically used: the pointwise inner product \eqref{inner} involves a more substantial change of metric near $\pi(M)$ (namely, $Z_{1}, \cdots, Z_{m}$, hence $\alpha_{1}, \cdots,\alpha_{m}$, are declared orthonormal). In particular, the domain of the adjoint of $\overline{\partial}$ changes. For a discussion of these matters, see for example \cite{CelikStraube08}, p. 175--176.

For a $(0,q)$-from $u$ and a function $\sigma$ equal to one on $b\Omega \cap \pi(M)$ and supported where $\chi_{1}$ equals one, calculation of the Kohn-Morrey-H\"{o}rmander formula in local coordinates and dropping the nonnegative boundary term gives (see for example \cite{Shaw85}, formula (3.17), \cite{Ahn07}, Proposition 2.1 with $s=0$, \cite{Zampieri08}, Proposition 1.9.1 with $q_{0}=0$)
\begin{multline}\label{KMHloc}
\int_{\Omega}\sideset{}{'}\sum_{|K|=q-1}\sum_{j,k=1}^{m}(\varphi)_{jk}(\sigma u)_{jK}\overline{(\sigma u)_{kK}}\,e^{-\varphi} \leq 2 \left (\|\overline{\partial}(\sigma u)\|_{\varphi}^{2} + \|\overline{\partial}^{*}_{\varphi}(\sigma u)\|_{\varphi}^{2} \right ) + C\|\sigma u\|_{\varphi}^{2} \; ,\\
u \in C^{\infty}_{(0,q)}(\overline{\Omega}) \cap dom(\overline{\partial}_{\varphi}^{*}) \; ,
\end{multline}
where the matrix $(\varphi)_{jk}$, $1 \leq j,k \leq m$, is defined via $\partial\overline{\partial}\varphi = \sum_{j,k=1}^{m}(\varphi)_{jk}\alpha_{j}\wedge \overline{\alpha_{k}}$ on the support of $\sigma$, and $C$ is a constant that does not depend on $\varphi$. 

Next, we choose $\varphi(z)=\varphi_{A}(z)$. For $A>0$, let $\lambda_{A}$ be the function from the definition of property($P_{q}$) for $M$. Pick a cutoff function $\chi_{A}$, $0 \leq \chi_{A} \leq 1$, that equals one in a neighborhood of $b\Omega \cap \pi(M)$, is supported where $\chi_{1}$ equals one (so $\chi_{1}$ and $\chi_{A}$ are nested), and whose support is also small enough so that the composition with the graphing function for $\widetilde{M}$ is defined on this support. This last condition makes $\chi_{A}$ dependent on $A$, whence the subscript. Set
\begin{equation}\label{Pq-function}
\varphi_{A}(z) = \chi_{A}(z)\lambda_{A}(z_{1}, \cdots, z_{m}, H_{1}(z_{1}, \cdots, z_{m}), \cdots, H_{n-m}(z_{1}, \cdots, z_{m}))\; ;
\end{equation}
that is, $\varphi_{A}$ is the pullback, under the graphing function for $\widetilde{M}$, of $\lambda_{A}$, times the cutoff factor $\chi_{A}$ (and so is defined on all of $\overline{\Omega}$). We claim that for $z$ in the set where $\chi_{A}$ equals one, the matrix $(\varphi_{A})_{jk}(z)$ satisfies \eqref{Pq2a}. Let $\underline{t}^{1}, \cdots, \underline{t}^{q}$ be orthonormal in $\mathbb{C}^{m}$. Then, by the definition of $(\varphi_{A})_{jk}$ (note that $(\alpha_{j_{1}} \wedge \overline{\alpha_{k_{1}}})(Z_{j_{2}},\overline{Z_{k_{2}}}) 
= (1/2)\delta_{j_{1}j_{2}}\delta_{k_{1}k_{2}}$ ),
\begin{multline}\label{hessiancomp}
\sum_{s=1}^{q}\sum_{j,k=1}^{m}(\varphi_{A})_{jk}(z)(\underline{t}^{s})_{j}\overline{(\underline{t}^{s})_{k}} = 2\sum_{s=1}^{q}(\partial\overline{\partial}\varphi_{A})(z)\left (\sum_{j=1}^{m}(\underline{t}^{s})_{j}Z_{j}(z)\,, \,\sum_{k=1}^{m}\overline{(\underline{t}^{s})_{k}}\;\overline{Z_{k}(z)} \right ) \\
= 2\sum_{s=1}^{q}(\partial\overline{\partial}\lambda_{A})(\zeta)\left (\sum_{j=1}^{m}(\underline{t}^{s})_{j}L_{j}(\zeta)\,, \,\overline{\sum_{k=1}^{m}(\underline{t}^{s})_{k}\;L_{k}(\zeta)} \;\right ) \;\geq \; A \; ,
\end{multline}
where $\zeta = (z_{1}, \cdots, z_{m}, H_{1}(z), \cdots, H_{n-m}(z))$. In the last inequality in \eqref{hessiancomp}, we have used \eqref{Pq2a} and the fact that the vectors $\sum_{j=1}^{m}(\underline{t}^{s})_{j}L_{j}(\zeta)$, $1 \leq s \leq q$, are orthonormal in $\mathbb{C}^{n}$ (because $\underline{t}^{s}$, $1 \leq s \leq q$, are orthonormal in $\mathbb{C}^{m}$, and $L_{j}$, $1 \leq j \leq m$, are orthonormal in $\mathbb{C}^{n}$). This establishes the claim.

With $\varphi_{A}$ chosen, we establish the following estimate: 
\begin{equation}\label{estimate1}
\|u\|_{\varphi_{A}}^{2} \leq  \frac{C}{A} \left (\|\overline{\partial}u\|_{\varphi_{A}}^{2} + \|\overline{\partial}_{\varphi_{A}}^{*}u\|_{\varphi_{A}}^{2} \right ) + C_{A}\|u\|_{-1,\varphi_{A}} \; , \; u \in dom(\overline{\partial}_{\varphi_{A}}) \cap dom(\overline{\partial}_{\varphi_{A}}^{*}) \; ,
\end{equation}
for $A$ big enough. Here, the subscript $\varphi_{A}$ refers to the norms given by \eqref{inner2} (with $\varphi = \varphi_{A}$) and the corresponding adjoint of $\overline{\partial}$. Note that in view of the uniform bounds on $\varphi_{A}$ (and the uniform equivalence of the pointwise inner products \eqref{inner}), these norms are equivalent to the Euclidean norm on forms, uniformly in $A$. Restrict the cutoff function $\sigma$ from above to be supported on the set where $\chi_{A}$ equals one (and still let it equal one in a small neighborhood of $b\Omega \cap \pi(M)$), and denote it by $\sigma_{A}$. Then both the support of $(1-\sigma_{A})$ and the suport of $\nabla\sigma_{A}$ intersect the boundary of $\Omega$ in the set of strictly pseudoconvex boundary points. Near such points, there is a subelliptic $(1/2)$-estimate; this can be seen by checking that the usual proofs based on the boundary term in the Kohn-Morrey-H\"{o}rmander formula still apply when the formula is written in terms of local coordinates (\cite{FollandKohn72, ChenShaw01, Straube10} (alternatively, there is a general result by Sweeney \cite{Sweeney76} to the effect that subellipticity is independent of the metric; this is however considerably more sophisticted than what is needed here). Combining these estimates with interpolation of Sobolev norms and setting $E_{A}:= (\max|\nabla\sigma_{A}|+ 1)$ gives
\begin{multline}\label{eq1}
\|(1-\sigma_{A})u\|_{\varphi_{A}}^{2} \leq \frac{1}{AE_{A}}\|(1-\sigma_{A})u\|_{1/2,\varphi_{A}}^{2} + C_{A}\|(1-\sigma_{A})u\|_{-1,\varphi_{A}}^{2} \\
\leq \frac{C}{AE_{A}} \left (\|\overline{\partial}\left ((1-\sigma_{A})u \right )\|_{\varphi_{A}}^{2} + \|\overline{\partial}_{\varphi_{A}}^{*} \left ((1-\sigma_{A})u \right )\|_{\varphi_{A}}^{2} \right ) + C_{A}\|(1-\sigma_{A})u\|_{-1,\varphi_{A}}^{2}  \\
\leq \frac{C}{A} \left (\|\overline{\partial}u\|_{\varphi_{A}}^{2} + \|\overline{\partial}_{\varphi_{A}}^{*}u\|_{\varphi_{A}}^{2} + \|u\|_{\varphi_{A}}^{2} \right ) + C_{A}\|u\|_{-1,\varphi_{A}}^{2} \; .
\end{multline}
(We note that as usual, constants are allowed to change from one occurrence to the next.) On the other hand, for $\sigma_{A} u$, we obtain from \eqref{KMHloc}, \eqref{hessiancomp}, and the equivalence of \eqref{Pq2} and \eqref{Pq2a}, and for $u \in C^{\infty}_{(0,q)}(\overline{\Omega}) \cap \cap dom(\overline{\partial}_{\varphi_{A}}^{*})$:
\begin{multline}\label{eq2}
\|\sigma_{A} u\|_{\varphi_{A}}^{2} \leq \frac{2}{A} \left (\|\overline{\partial}(\sigma_{A} u)\|_{\varphi_{A}}^{2} + \|\overline{\partial}^{*}_{\varphi_{A}}(\sigma_{A} u)\|_{\varphi_{A}}^{2} \right ) + \frac{C}{A}\|\sigma_{A} u\|_{\varphi_{A}}^{2} \\
\leq \frac{2}{A} \left (\|\overline{\partial}u\|_{\varphi_{A}}^{2} + \|\overline{\partial}^{*}_{\varphi_{A}}u\|_{\varphi_{A}}^{2}  + \|(\nabla\sigma_{A})u\|_{\varphi_{A}}^{2} \right ) + \frac{C}{A}\|u\|_{\varphi_{A}}^{2} \\
\leq \frac{C}{A} \left (\|\overline{\partial}u\|_{\varphi_{A}}^{2} + \|\overline{\partial}_{\varphi_{A}}^{*}u\|_{\varphi_{A}}^{2} + \|u\|_{\varphi_{A}}^{2} \right ) + C_{A}\|u\|_{-1,\varphi_{A}}^{2}\; ,
\end{multline}
where $(\nabla\sigma_{A})u$ stands for terms of the form derivative of $\sigma_{A}$ times $u$. The last inequality in \eqref{eq2} follows by an interpolation of Sobolev norms for these terms similar to the one used in \eqref{eq1}.

Adding estimates \eqref{eq1} and \eqref{eq2} and absorbing the $\|u\|_{\varphi_{A}}^{2}$ terms gives \eqref{estimate1}, for $A$ big enough, and for $u \in C^{\infty}_{(0,q)}(\overline{\Omega}) \cap dom(\overline{\partial}_{\varphi_{A}}^{*})$. Finally, when $u$ is only in $ dom(\overline{\partial}_{\varphi_{A}}) \cap  dom(\overline{\partial}_{\varphi_{A}}^{*})$, we invoke the density in the graph norm of forms smooth up to the boundary. This density can be established in the same way as for the Euclidean norm: one can check that for example the proof in \cite{Straube10}, Proposition 2.3, also work for the metrics considered here. This establishes \eqref{estimate1}.

Having established \eqref{estimate1} on $\Omega$, we now turn to $\Omega^{+}:=\mathbb{B} \setminus \overline{\Omega}$. The argument here consists of an adaption of the proof of Proposition 3.1 in \cite{RaichStraube08} (to the effect that the standard $\overline{\partial}$-Neumann operator $N_{q}^{\Omega^{+}}$ is compact if $b\Omega$ satisfies property($P_{m-1-q}$)). The vector fields $Z_{j}$, $1 \leq j \leq m$, defined only in a one sided neighborhood of $\pi(M)$, can be extended into a full neighborhood of $b\Omega \cap \pi(M)$. The forms $\alpha_{1}, \cdots, \alpha_{m}$ then also extend as dual to $Z_{1}, \cdots, Z_{m}$. Now define a pointwise inner product and then an $\mathcal{L}^{2}$ inner product on $\Omega^{+}$ in analogy to \eqref{inner} and \eqref{inner2}. Define $\varphi_{A}$ as in \eqref{Pq-function}, but with $\lambda_{A}$ now the \emph{negative} of the function that comes from the definition of property($P_{m-1-q}$) of $M$. Then $\varphi_{A}$ is defined in a one sided neighborhood of $\pi(M)$, and $(-\varphi_{A})_{jk}$ satisfies \eqref{hessiancomp} with $q$ replaced by $(m-1-q)$. We can extend $\varphi_{A}$ into a full neighborhood of $b\Omega \cap \pi(M)$ (depending on $A$) so that \eqref{hessiancomp} (with $-\varphi_{A}$ in place of $\varphi_{A}$ and $q$ replaced by $(m-1-q)$) still holds there (possibly replacing $A$ by $A/2$, which is immaterial), and also keeping a uniform bound. We also choose $\sigma_{A}$ as above, equal to one in a neighborhood of $b\Omega \cap \pi(M)$ and supported near $b\Omega \cap \pi(M)$ as well.

For $u \in C^{\infty}_{(0,q)}(\overline{\Omega^{+}}) \cap dom(\overline{\partial}_{\varphi_{A}}^{*})$, we again split $u$ into $\sigma_{A}u + (1-\sigma_{A})u$. The estimation of $(1-\sigma_{A})u$ is the same as above, except that on the strictly pseudoconcave part of the boundary of $\Omega^{+}$, one has to invoke condition $Z(q)$ (\cite{FollandKohn72}, p.57) for the subelliptic $1/2$-estimate. Namely, for $z \in b\Omega \setminus (b\Omega \cap \pi(M))$, viewed as a boundary point of $\Omega^{+}$, the Levi form has $(m-1) \geq (q+1)$ negative eigenvalues. By \cite{Hormander65}, Theorem 3.25 or \cite{FollandKohn72}, Theorem 3.2.2, this implies to a subelliptic $1/2$-estimate. Actually, in both these references, the estimate is formulated not in terms of the $1/2$-norm, but rather in terms of the $\mathcal{L}^{2}$-norm on the boundary. But the $1/2$-estimate follows from this via elliptic theory: $\|u\|_{1/2}$ is dominated by the boundary $\mathcal{L}^{2}$-norm plus the $(-1)$-norm of $\Box_{\varphi_{A}}u$, which in turn is dominated by $\|\overline{\partial}u\| + \|\overline{\partial}_{\varphi_{A}}^{*}u\|$ (see for example \cite{Straube10}, (proof of) Proposition 3.1, compare also \cite{KohnNirenberg65}, Theorem 5).

To estimate $\sigma_{A}u$, we start by replacing \eqref{KMHloc} with the following estimate (which now also includes the boundary terms)
\begin{multline}\label{KMHloc2}
\int_{\Omega^{+}}\left (\sideset{}{'}\sum_{|K|=q-1}\sum_{j,k=1}^{m}(\varphi_{A})_{jk}(\sigma_{A}u)_{jK}\overline{(\sigma_{A}u)_{kK}} \right )e^{-\varphi_{A}} - \int_{\Omega^{+}}\sideset{}{'}\sum_{|J|=q}\left (\sum_{j \leq m-1}(\varphi_{A})_{jj} \right )|\sigma_{A}u_{J}|^{2}\;e^{-\varphi_{A}} \\
+ \int_{b\Omega^{+}}\left (\sideset{}{'}\sum_{|K|=q-1}\sum_{j,k=1}^{m}\rho_{jk}(\sigma_{A}u)_{jK}\overline{(\sigma_{A}u)_{kK}} \right ) e^{-\varphi_{A}} - \int_{b\Omega^{+}}\sideset{}{'}\sum_{|J|=q}\left (\sum_{j \leq m-1}\rho_{jj} \right )|\sigma_{A}u_{J}|^{2}\; e^{-\varphi_{A}} \\
\leq 2 \left (\|\overline{\partial}(\sigma_{A}u)\|_{\varphi_{A}}^{2} + \|\overline{\partial}_{\varphi_{A}}^{*}(\sigma_{A}u)\|_{\varphi_{A}}^{2} \right ) + C\|\sigma_{A}u\|_{\varphi_{A}}^{2} \; , \; u \in C^{\infty}_{(0,q)}(\overline{\Omega^{+}}) \cap dom(\overline{\partial}_{\varphi_{A}}^{*}) \; ,
\end{multline}
where $C$ does not depend on $A$, and $\rho$ is a normalized defining function for $b\Omega^{+}$ (say, near $b\Omega \cap \pi(M)$). \eqref{KMHloc2} is in \cite{Ahn07}, Proposition 2.1 (with $s=(m-1)$) and in \cite{Zampieri08}, Proposition 1.9.1 (with $q_{0}=m-1$); it is implicit already in \cite{Shaw85}, formula (3.23). Now the argument follows verbatim the one in \cite{RaichStraube08}, starting with inequality (15) there, and we outline it only briefly, referring the reader to \cite{RaichStraube08} for details. The first observation is that the boundary terms can be collected into
\begin{equation}\label{boundaryint}
\int_{b\Omega^{+}}\left (\sideset{}{'}\sum_{|K|=q-1}\sum_{j,k=1}^{m-1}\left (\rho_{jk}-\frac{1}{q}\left (\sum_{l=1}^{m-1}\rho_{ll}\right )\delta_{jk} \right ) (\sigma_{A}u)_{jK}\overline{(\sigma_{A}u)_{kK}} \right )e^{-\varphi_{A}} \; ,
\end{equation}
see (16) in \cite{RaichStraube08}. The sum over $|K|=(q-1)$ is also only over $K$ that do not contain $m$ (since then $u_{jK}=0$, because $u \in dom(\overline{\partial}_{\varphi_{A}}^{*})$). Therefore, the sum in the integrand equals at least $|\sigma_{A}u|^{2}$ times the sum of the smallest $q$ eigenvalues of the Hermitian matrix in the integrand. This sum equals minus the trace plus the sum of the smallest $q$ eigenvalues of $(\rho_{jk})_{jk=1}^{m-1}$, so equals minus the sum of $m-1-q$ eigenvalues of $(\rho_{jk})_{jk=1}^{m-1}$. Because the complex Hessian of $\rho$ is negative semidefinite on the complex tangent space, this sum is nonnegative. Thus the boundary terms give a nonnegative contribution to the left hand side of \eqref{KMHloc2} and can be dropped from the inequality. In the terms involving integration over the interior, restricting the sums over $K$ with $m \notin K$, $j,k < m$, and $J$ with $m \notin J$ makes an error that is controlled as follows. When $m \in J$, $u_{J}$ vanishes on the boundary (because $u \in dom(\overline{\partial}_{\varphi_{A}}^{*})$), and  $\|\overline{\partial}\overline{\partial}_{\varphi_{A}}^{*}(\sigma_{A}u) + \overline{\partial}_{\varphi_{A}}^{*}\overline{\partial}(\sigma_{A}u)\|_{-1}$ controls  $\|(\sigma_{A}u)_{J}\|_{1}$. This is well known for the Euclidean metric; for a proof see \cite{Straube10}, Lemma 2.12. For the case of our metric, the same argument works because the second order part of $\Box_{\varphi_{A}}$ still acts diagonally (in the chosen basis) via an elliptic operator. This is essentially the same computation as in the Euclidean case, but done in local coordinates. The constant in the resulting estimate depends on $A$; via interpolation of Sobolev norms, we obtain control of $\|(\sigma_{A}u)_{J}\|$ in terms of the right hand side of \eqref{KMHloc2}, modulo adding $\|\sigma_{A}u\|_{-1}$. The resulting estimate (corresponding to (18) in \cite{RaichStraube08}) is:
\begin{multline}\label{RS18}
\int_{\Omega^{+}}\left (\sideset{}{''}\sum_{|K|=q-1}\sum_{j,k=1}^{m-1}\left ((\varphi_{A})_{jk}-\frac{1}{q}\left (\sum_{l=1}^{m-1}(\varphi_{A})_{ll}\right )\delta_{jk} \right ) (\sigma_{A}u)_{jK}\overline{(\sigma_{A}u)_{kK}} \right ) e^{-\varphi_{A}} \\
\leq C \left (\|\overline{\partial}(\sigma_{A}u)\|_{\varphi_{A}}^{2} + \|\overline{\partial}_{\varphi_{A}}^{*}(\sigma_{A}u)\|_{\varphi_{A}}^{2}  + \|\sigma_{A}u\|_{\varphi_{A}}^{2} \right ) + C_{A}\|\sigma_{A}u\|_{-1, \varphi_{A}}^{2} \; ,
\end{multline}
where $\sum''$ denotes summation over strictly increasing multi-indices that do not contain $m$. Arguing as above, we obtain that the integrand in the left hand side of \eqref{RS18} is at least $|\sigma_{A}u|^{2}$ times minus the sum of $(m-1-q)$ eigenvalues of $((\varphi_{A})_{jk})_{j,k=1}^{m-1}$. This sum is at least equal to the sum of the smallest $(m-1-q)$ eigenvalues of $((-\varphi_{A})_{jk})_{j,k=1}^{m}$, which in turn is at least $A$ (this is why we chose the negative sign in the definition of $\lambda_{A}$). Choosing $A$ big enough and absorbing the $\|\sigma_{A}u\|_{\varphi_{A}}^{2}$ term gives the desired estimate
\begin{equation}\label{eq5}
\|\sigma_{A}u\|_{\varphi_{A}}^{2} \leq \frac{C}{A} \left (\|\overline{\partial}(\sigma_{A}u)\|_{\varphi_{A}}^{2} + \|\overline{\partial}_{\varphi_{A}}^{*}(\sigma_{A}u)\|_{\varphi_{A}}^{2}  \right ) + C_{A}\|\sigma_{A}u\|_{-1, \varphi_{A}}^{2} \; , 
\end{equation}
with $C$ independent of $A$. Again invoking the density of forms smooth up to the boundary in the graph norm (compare the remark in the paragraph following \eqref{eq2}), \eqref{eq5} carries over to $u \in dom(\overline{\partial}) \cap dom(\overline{\partial}_{\varphi_{A}}^{*})$.

Combining \eqref{eq5} with the estimate for $\|(1-\sigma_{A})u\|$ gives \eqref{estimate1} on $\Omega^{+}$ as well. 

From now on, the argument follows \cite{RaichStraube08} essentially verbatim. First, \eqref{estimate1} on $\Omega$ and $\Omega^{+}$ implies compactness of the $\overline{\partial}$-Neumann operators $N_{q}^{\Omega}$ and $N_{q}^{\Omega^{+}}$, respectively. This is shown in \cite{RaichStraube08} on pages 772--773. The fact that here the unweighted norm is not the Euclidean norm (and consequently the domain of the adjoint of $\overline{\partial}$ may change, see \cite{CelikStraube08}) does not matter; what does is that it is comparable to the Euclidean norm, so that the norms $\|\cdot\|_{\varphi_{A}}$ are also, and \emph{uniformly} so with respect to $A$. The argument for $\Omega$ is slightly simpler in that the space of harmonic forms is trivial. (This actually turns out to be the case also on $\Omega^{+}$ ($1 \leq q \leq (m-2)$), as shown recently in \cite{Shaw10}.) Finally, compactness of both $N_{q}^{\Omega}$ and $N_{q}^{\Omega^{+}}$ implies the desired compactness estimate on $(0,q)$-forms on $b\Omega$; this is shown in \cite{RaichStraube08}, section 4.

\vspace{0.15in}
\emph{Remark 3}: The fact that property($P_{q}$) is not (known to be) invariant under, say, local biholomorphisms smooth up to the boundary is unsatisfactory from the perspective of the property as a sufficient condition for compactness of $N_{q}$: the latter \emph{is} invariant (pullbacks of compact solution operators for $\overline{\partial}$ on the image domain produce compact solution operators on the source domain, see \cite{Straube10}, proof of the implication $(i) \Rightarrow (ii)$ in Theorem 4.26). This invariance may be viewed as a special case of invariance under a change of metric, by pulling back the metric from the image domain to the source domain, as in the proof of Theorem \ref{main} (see \cite{CelikStraube08} for the invariance of compactness under a change of metric, among metrics smooth up to the boundary). The part of the proof of Theorem \ref{main} where we worked on $\Omega$ shows in particular that it suffices for compactness of $N_{q}$ that there exist some metric (smooth up to the boundary) with respect to which property($P_{q}$) holds (in the sense of \eqref{hessiancomp}). Likewise, the work for $\Omega^{+}$ shows that such an invariant version of Proposition 3.1 in \cite{RaichStraube08} holds for $\Omega^{+}$. As a result, there is an invariant version of Theorem 1.4 in \cite{RaichStraube08}; this is the `slight extension' referred to in the abstract. 

\vspace{0.15in}
\emph{Remark 4}: In light of the comments in the previous remark, and the invariance result in \cite{CelikStraube08}, it seems appropriate, for the time being, to (re)formulate property($P_{q}$) invariantly when $q>1$. That is, there should exist a metric (smooth up to the boundary) so that $P_{q}$ holds with respect to this metric, in the sense of \eqref{hessiancomp}. 

\vspace{0.15in}
\emph{Remark 5}: A slightly different approach to the proof of Theorem \ref{main} is as follows. Instead of establishing compactness of $N_{q}^{\Omega}$ and $N_{q}^{\Omega^{+}}$, one establishes  compactness of $N_{q}^{\Omega}$ and $N_{m-1-q}^{\Omega}$ (from properties($P_{q}$) and ($P_{m-1-q}$) of $M$, respectively; so this is only one proof). Compactness at these two symmetric levels in the interior implies compactness on the boundary at the level of $(0,q)$-forms ($1 \leq q \leq (m-2)$). This implication is implicit in \cite{Kohn02} and explicit (within a general context) in \cite{Khanh10}. On the other hand, the argument for the annular domain $\Omega^{+}$ used above is of interest in its own right.

\vspace{0.15in}
\emph{Remark 6}: For emphasis, we point out the following observation which was used in the proof of Theorem \ref{main}. $\overline{\partial}_{b}$ (defined extrinsically) does not commute with pullbacks under CR-equivalences between CR-submanifolds of $\mathbb{C}^{n}$. Consequently, results about the complex associated with a particular CR-submanifold do not automatically transfer to the complex on a CR-equivalent submanifold. Nonetheless, solvability for $\overline{\partial}_{b}$ is preserved (\cite{Boggess91}, Corollary 2, section 9.2), and so is compactness: the compactness estimate \eqref{compest} is preserved under CR-equivalences. This is immediate from the properties \eqref{push} -- \eqref{commute*} of the modified pullback $\hat{\pi}$.

\section{Two potential theoretic conditions}\label{potential}
In \cite{Raich10}, property(CR-$P_{q}$) is introduced (for an orientable $M$), and it is shown that (CR-$P_{q}$) and (CR-$P_{m-1-q}$) are sufficient for compactness of $G_{q}$. It is also shown that ($P_{q}$) implies CR-($P_{q}$). We show here that the two are actually equivalent. In fact, both are equivalent to a version of ($P_{q}$) where the condition on the Hessian is required only in complex tangential directions that lie in the nullspace of the Levi form. In the context of ($\widetilde{P}_{q}$) for the boundary of a domain, this was observed in \cite{Celik08} (Proposition 1; compare also \cite{Sahutoglu09}, Corollary 2 and the remark that follows).

We recall the definition of property(CR-$P_{q}$) from \cite{Raich10}. $M$ is a smooth orientable CR-submanifold of $\mathbb{C}^{n}$ of hypersurface type. Because $M$ is orientable, the (real) vector field $T$ from section \ref{CR} can be chosen globally. Denote by $\eta$ the $1$-form on $M$ that annihilates $T^{1,0}(M) \oplus T^{0,1}(M)$ and that is constant one on $T$. Then the Levi form \eqref{levi} equals $\langle\eta, (1/2i)[L, \overline{L}]\rangle = (i/2)\langle d\eta, L\wedge \overline{L}\rangle$, where $L$ is a local section of $T^{1,0}(M)$; the equality follows from the definition of the exterior derivative and the fact that $\langle \eta, L\rangle \equiv \langle \eta, \overline{L}\rangle \equiv 0$. Replacing $T$ by $-T$ (hence $\eta$ by $-\eta$) if necessary, we may assume that the Levi form is positive definite. Let $\gamma := (i/2)\eta$. Then $M$ is said to satisfy property(CR-$P_{q}$) if for every $A>0$ there exist a function $\lambda_{A}$ defined in a neighborhood $U_{\lambda}$ of $M$ with $0 \leq \lambda_{A} \leq 1$ and a constant $A_{0}$ such that
\begin{equation}\label{CRPq}
\sum_{j=1}^{q}\left \langle \frac{1}{2} (\partial_{b}\overline{\partial}_{b}\lambda_{A} - \overline{\partial}_{b}\partial_{b}\lambda_{A} )+A_{0}d\gamma\;, \,L_{j} \wedge \overline{L_{j}} \right \rangle \; \geq \; A \; 
\end{equation}
at all points $z \in M$ and any set $\{L_{j}\}_{j=1}^{q}$ of orthonormal vectors in  $T^{1,0}_{z}(M)$ (\cite{Raich10}, Definition 2.6).

We say that $M$ satisfies property($P_{q}$) on the nullspace of the Levi form if for every $A>0$ there exists a (smooth) function $\lambda_{A}$ defined in a neighborhood $U_{A}$ of $M$ with $0 \leq \lambda_{A} \leq 1$ and 
\begin{equation}\label{Pqnull}
\sum_{s=1}^{q}\sum_{j,k=1}^{n}\frac{\partial^{2}\lambda_{A}(z)}{\partial z_{j}\partial\overline{z_{k}}}(L_{s})_{j}\overline{(L_{s})_{k}} \,\geq \, A \; ,
\end{equation}
whenever $L_{1}, \cdots, L_{q}$ are orthonormal and are contained in the nullspace $\mathcal{N}_{z}$ of the Levi form of $M$ at $z \in M$.
\begin{proposition}\label{equivalent}
Let $M \subset \mathbb{C}^{n}$ be a smooth compact pseudoconvex CR-submanifold of hypersurface type, let $1 \leq q \leq m-1$, where $(m-1)$ is the CR-dimension of $M$. Then the following conditions are equivalent:

\noindent (i) $M$ satisfies property($P_{q}$) on the nullspace of the Levi form;

\noindent (ii) $M$ satisfies property($P_{q}$).

\noindent For orientable $M$, (i) and (ii) are also equivalent to

\noindent (iii) $M$ satisfies property(CR-$P_{q}$).
\end{proposition}
The implication $(ii) \Rightarrow (iii)$ is in Corollary 1.2 in \cite{Raich10}. Note that orientability is required for the definition of (CR-$P_{q}$) in $(iii)$.
\begin{proof}
That (ii) implies (i) is a trivial consequence of the definitions. The implication $(iii) \Rightarrow (i)$ is an immediate consequence of \cite{Raich10}, Proposition 3.2. By that proposition, with $\{L_{s}\}$ as above (so that $\langle d\gamma, L_{s}\wedge \overline{L_{s}}\rangle = 0$), $1 \leq s \leq q$,
\begin{multline}\label{raichprop}
\frac{1}{2}\sum_{s=1}^{q}\sum_{j,k=1}^{n}\frac{\partial^{2}\lambda_{A}(z)}{\partial z_{j}\partial\overline{z_{k}}}(L_{s})_{j}\overline{(L_{s})_{k}}\; = \;\sum_{s=1}^{q}\left \langle\partial\overline{\partial}\lambda_{A}, L_{s}\wedge \overline{L_{s}} \right \rangle \\
= \sum_{s=1}^{q}\left \langle \frac{1}{2} (\partial_{b}\overline{\partial}_{b}\lambda_{A} - \overline{\partial}_{b}\partial_{b}\lambda_{A} )\,, \,L_{s} \wedge \overline{L_{s}} \right \rangle 
+ \frac{1}{2}\nu(\lambda_{A})\sum_{s=1}^{q} \langle d\gamma, L_{s}\wedge \overline{L_{s}}\rangle \\
= \sum_{s=1}^{q}\left \langle \frac{1}{2} (\partial_{b}\overline{\partial}_{b}\lambda_{A} - \overline{\partial}_{b}\partial_{b}\lambda_{A} )+A_{0}d\gamma\;, \,L_{s} \wedge \overline{L_{s}} \right \rangle \; \geq \; A \; ,
\end{multline}
where $\nu(\lambda_{A})$ is the derivative of $\lambda_{A}$ in the direction of $\nu$, the `normal' to $M$ given by $J$ (the complex structure map of $\mathbb{C}^{n}$) times the unit tangent field to $M$ that is perpendicular to $T^{\mathbb{C}}(M)$. The second equality is Proposition 3.2 in \cite{Raich10}, and the last inequality holds if $\lambda_{A}$ is the function from the definition of property(CR-$P_{q}$).

We next show that (i) implies (iii). The argument is standard. Denote by $\mathbb{B}^{n}$ the unit ball in $\mathbb{C}^{n}$, by $S$ the compact subset of $M \times b\mathbb{B}^{n} \times \cdots \times b\mathbb{B}^{n}$ ($q$-fold) given by $\{(z, L_{1}, \cdots, L_{q})\, |\, L_{s} \in T^{1,0}_{z}(M), 1 \leq s \leq q\,,\, L_{1}, \cdots, L_{q} \;\text{are orthonormal}\}$, and by $K$ the compact subset of $S$ given by the requirement that $L_{s} \in \mathcal{N}_{z}$, $1 \leq s \leq q$. Let $A>0$ and let $\lambda_{A}$ be the function from (i). Then on $K$, again because $\langle d\gamma,L_{s}\wedge \overline{L_{s}}\rangle = 0$, $1 \leq s \leq q$, 
\begin{equation}\label{eq10}
\sum_{j=1}^{q}\left \langle \frac{1}{2} (\partial_{b}\overline{\partial}_{b}\lambda_{A} - \overline{\partial}_{b}\partial_{b}\lambda_{A} ), \,L_{j} \wedge \overline{L_{j}} \right \rangle \; 
= \sum_{s=1}^{q}\left \langle\partial\overline{\partial}\lambda_{A}\,,\,L_{s}\wedge \overline{L_{s}} \right \rangle \geq \; \frac{A}{2} \; 
\end{equation}
Choose a neighborhood $U$ of $K$ (in $S$) so that \eqref{eq10} holds with $(A/2 - 1)$ on the right hand side for $(z, L_{1}, \cdots, L_{q}) \in U$; this can be done by continuity. Next, choose a neighborhood $V \subset\subset U$ of $K$, and denote by $B_{1}$ and $B_{2}$ the minimum of the left hand side of \eqref{eq10} and the strictly positive minimum of $\sum_{s=1}^{q}\langle d\gamma, L_{s}\wedge \overline{L_{s}}\rangle$, respectively, for $(z, L_{1}, \cdots, L_{q}) \in S \setminus V$.
Then for $A_{0}>0$, we have on $S\setminus V$
\begin{equation}\label{eq11}
\sum_{j=1}^{q}\left \langle \frac{1}{2} (\partial_{b}\overline{\partial}_{b}\lambda_{A} - \overline{\partial}_{b}\partial_{b}\lambda_{A} )+A_{0}d\gamma\;, \,L_{j} \wedge \overline{L_{j}} \right \rangle \geq B_{1} + A_{0}B_{2} \geq A
\end{equation}
if $A_{0}$ is chosen sufficiently large. When $(z, L_{1}, \cdots, L_{q}) \in  U$, the left hand side of \eqref{eq11} equals at least $\sum_{s=1}^{q}\left \langle \frac{1}{2} (\partial_{b}\overline{\partial}_{b}\lambda_{A} - \overline{\partial}_{b}\partial_{b}\lambda_{A}), \,L_{s} \wedge \overline{L_{s}} \right \rangle$ (by pseudoconvexity, $\langle d\gamma,L_{s}\wedge \overline{L_{s}}\rangle \geq 0$), hence at least $(A/2 - 1)$. Since $A>0$ is arbitrary and $U \cup (S\setminus V) = S$, this and \eqref{eq11} show that $M$ satisfies property(CR-$P_{q}$).

It remains to show $(i) \Rightarrow (ii)$. The strategy is the same as in the previous argument, except that the correction function that makes things big away from $K$ has to be in the form of a Hessian. So we need to express the Levi form in terms of Hessians. It suffices to prove property($P_{q}$) locally (\cite{Sibony87b}, \cite{Straube10}).  Recall that here $M$ is not assumed orientable, and the Levi form is is only defined locally, and only up to sign. Let $\rho_{1}, \cdots, \rho_{d}$, $d=2n-2m+1$, be a set of defining functions for $M$, say near a point $P$, so that their gradients are orthonormal. Choose the field $T$ to be orthogonal to $T^{1,0}(M) \oplus T^{0,1}(M)$.
By \cite{Boggess91}, Theorem 1 in section 10.2, the Levi form is given, up to sign, by the expression
\begin{equation}\label{levihessian}
\sum_{l=1}^{d}\left (\sum_{j,k=1}^{n}\frac{\partial^{2}\rho_{l}}{\partial z_{j}\partial\overline{z_{k}}}w_{j}\overline{w_{k}} \right ) \left \langle \eta, J\nabla\rho_{l}\right\rangle = \sum_{l=1}^{d}\left (\sum_{j,k=1}^{n}\frac{\partial^{2}(\rho_{l}\langle\eta, J\nabla\rho_{l}\rangle)}{\partial z_{j}\partial\overline{z_{k}}}w_{j}\overline{w_{k}} \right ) \; ,
\end{equation}
where $L = \sum_{j=1}^{n}w_{j}(\partial/\partial z_{j}) \in T^{1,0}(M)$. Here we think of $\eta$ as being defined also on vectors not in the tangent space to $M$ by setting it equal to zero on $T(M)^{\perp}$ (so that the right hand side of \eqref{levihessian} is well defined). Moreover, this definition works (locally) near $M$ as usual by considering the submanifolds given by $\rho_{l}=\varepsilon_{l}$, $1 \leq l \leq d$, for  $\varepsilon_{l}$ small. The equality in \eqref{levihessian} then holds for $z \in M$ because $\rho_{l}(z)=0$, $1 \leq l \leq d$, and $\sum_{j=1}^{n}(\partial\rho_{l}/\partial z_{j})(z)w_{j} = 0$. We may assume that \eqref{levihessian} is nonnegative (again changing $T$ to $-T$ if necessary).

Let now $A>0$, and $\lambda_{A}$ the function form the definition of property($P_{q}$) on the null space of the Levi form. Then the same argument as above shows that if $B_{1}$ is chosen sufficiently large, the function $\lambda_{A}+B_{1}\sum_{l=1}^{d}\rho_{l}\langle\eta,J\nabla\rho_{l}\rangle$ has Hessian greater than, say, $(A-1)$ on vectors in $T^{1,0}(M)$, near $P$. Because the correction term vanishes at points of $M$, we can keep uniform (in $A$) bounds, say $-1/2$ and $3/2$, by shrinking the neighborhood $U_{A}$ ; the actual value of the bounds is immaterial. Finally, adding $B_{2}\sum_{l=1}^{d}\rho_{l}^{2}$ and choosing $B_{2}$ large enough produces a function whose Hessian is large on all vectors in $\mathbb{C}^{n}$ (compare for example \cite{Sibony87b}, proof of Lemme 1.13, \cite{Catlin84b}, proof of Theorem 2, or \cite{Straube10}, proof of Proposition 4.15). This proves that $M$ satisfies property($P_{q}$) locally, and as was noted above, this implies that $M$ satisfies $P_{q}$ globally.
\end{proof}
\vspace{0.1in}
\emph{Remark 7}: Proposition \ref{equivalent} is a manifestation of a general, albeit somewhat vague, principle to the effect that conditions for some regularity property of $\overline{\partial}$ or $\overline{\partial}_{b}$ involving the complex tangent space need only be imposed on the nullspace of the Levi form. In the context of compactness, compare again \cite{Celik08, Sahutoglu09}. In connection with sufficient conditions for global regularity, manifestations of this principle occur for example in \cite{Straube10}, Theorem 5.9 and Proposition 5.13.

\bigskip
\bigskip
\emph{Acknowledgement:} The author is indebted to Andrew Raich for useful comments on a preliminary version of this manuscript.

\bigskip
\providecommand{\bysame}{\leavevmode\hbox to3em{\hrulefill}\thinspace}


\begin{thebibliography}{10}

\bibitem{Ahn07}
Ahn, Heungju, Global boundary regularity for the $\overline\partial$-equation on $q$-pseudoconvex domains,  \emph{Math. Nachr.} \textbf{280}, no.~4 (2007), 343--350.

\bibitem{Amar84}
Amar, \'{E}., Cohomologie complexe et applications, \emph{J. London Math. Soc.} (2) \textbf{29}, no.1  (1984), 127--140.

\bibitem{BER}
Baouendi, M.~Salah, Ebenfelt Peter, and Preiss Rothschild, Linda, \emph{Real Submanifolds in Complex Space and Their Mappings}, Princeton University Press, Princeton, 1999.

\bibitem{Bell86}
Bell, Steve, Differentiability of the Bergman kernel and pseudolocal estimates, \emph{Math. Z.} \textbf{192}  (1986), 467--472.

\bibitem{BoasShaw86}
Boas, Harold P. and Shaw, Mei-Chi, Sobolev estimates for the Lewy operator on weakly pseudoconvex boundaries, \emph{Math. Ann.} \textbf{274}, no.2  (1986), 221--231.

\bibitem{BoasStraube91}
Boas, Harold P. and Straube, Emil J., Sobolev estimates for the complex Green operator on a class of weakly pseudoconvex boundaries, \emph{Comm. Partial Differential Equations} \textbf{16}, no. 10 (1991),  1573--1582.

\bibitem{Boggess91}
Boggess, Albert, \emph{CR-Manifolds and the Tangential Cauchy-Riemann Complex}, Studies in Advanced Mathematics, CRC Press 1991.

\bibitem{BdM75}
Boutet de Monvel, L. Int\'{e}gration des \'{e}quations de Cauchy-Riemann induites formelles, \emph{S\'{e}minaire Goulaouic-Lions-Schwartz 1974--1975; \'{E}quations Aux Deriv\'{e}es Partielles Lin\'{e}aires et Non Lin\'{e}aires}, Exp. No. 9, Centre Math., \'{E}cole Polytech., Paris, 1975. 

\bibitem{Burns79}
Burns, Daniel M., Jr., Global behavior of some tangential Cauchy-Riemann equations, in \emph{Partial Differential Equations and Geometry} (Park City, Utah, 1977),  pp. 51--56, Lecture Notes in Pure and Appl. Math., 48, Dekker, New York, 1979.

\bibitem{Catlin84b}
Catlin, David W., Global regularity of the {$\bar \partial
$}-{N}eumann problem. In \emph{Complex Analysis of Several Variables} (Madison 1982). Proc. Sympos. Pure Math. 41, Amer. Math. Soc., Providence 1984, 39--49.

\bibitem{Celik08}
\c{C}elik, Mehmet, Contributions to the compactness theory of the $\overline{\partial}$-Neumann operator, Diss. Texas A\&M University, 2008.

\bibitem{CelikStraube08}
\c{C}elik, Mehmet and Straube, Emil J., Observations regarding compactness in the $\overline{\partial}$-Neumann problem, \emph{Complex Variables and Elliptic Equations} \textbf{54}, nos. 3--4 (2009), 173--186.

\bibitem{ChenShaw01}
Chen, So-Chin and Shaw, Mei-Chi, \emph{Partial Differential Equations in
Several Complex Variables}, Studies in Advanced Mathematics 19, Amer. Math. Soc./International Press, 2001.

\bibitem{Davies95}
Davies, E.~B., \emph{Spectral Theory and Differential Operators}, Cambridge Studies in Advanced Mathematics, nr. 42, Cambridge Univ. Press, 1995.

\bibitem{FollandKohn72}
Folland, G.~B. and Kohn, J.~J., \emph{The Neumann Problem for the
Cauchy-Riemann Complex}, Annals of Mathematics Studies 75,
Princeton University Press, 1972.

\bibitem{FuStraube98}
Fu, Siqi and Straube, Emil J., Compactness of the
{$\overline\partial$}-{N}eumann problem on convex domains, \emph{J. Funct.
Anal.} \textbf{159} (1998), 629--641.

\bibitem{FuStraube99}
\bysame, Compactness in the
{$\overline\partial$}-{N}eumann problem. In \emph{Complex Analysis and
Geometry} (Columbus 1999). Ohio State Univ. Math. Res. Inst.
Publ. 9, de Gruyter, Berlin 2001, 141--160.

\bibitem{GaySebbar85}
Gay, Roger and Sebbar, Ahmed, Division et extension dans l'alg\`{e}bre $A^\infty(\Omega)$ d'un ouvert pseudo-convexe \`{a} bord lisse de $C^n$, \emph{Math. Z.} \textbf{189}, no. 3 (1985), 421--447. 

\bibitem{HarRaich10}
Harrington Phillip S. and Raich, Andrew, Regularity results for $\overline{\partial}_{b}$ on CR-manifolds of hypersurface type, \emph{Commun. Partial Diff. Equations}, to appear.

\bibitem{HarveyLawson75}
Harvey, F.~Reese and Lawson, H.~Blaine, On boundaries of complex analytic varieties, I, \emph{Ann. Math.} \textbf{102} (1975), 223--290.

\bibitem{Hormander65}
H\"{o}rmander, L., $L^{2}$ estimates and existence theorems for the $\overline{\partial}$ operator, \emph{Acta Math.} \textbf{113} (1965), 89--152.

\bibitem{Khanh10}
Khanh, Tran V., A general method of weights in the $\overline{\partial}$-Neumann problem, preprint 2010, arXiv:1001.5093. 

\bibitem{Koenig04}
Koenig, Kenneth D., A parametrix for the $\overline{\partial}$-Neumann problem on pseudoconvex domains of finite type, \emph{J. Func. Anal.} \textbf{216} (2004), 243--302.

\bibitem{Kohn81}
Kohn, Joseph J., Boundary regularity of $\overline{\partial}$, in \emph{Recent Developments in Several Complex Variables}, 243--260, Annals of Math, Studies \textbf{100}, Princeton University Press, Princeton, 1981.

\bibitem{Kohn86}
\bysame, The range of the tangential Cauchy-Riemann operator, \emph{Duke Math. J.} \textbf{53}, no.2  (1986), 525--545.

\bibitem{Kohn90}
\bysame, Embedding of CR manifolds, in \emph{Partial Differential Equations and Related Subjects} (Trento, 1990), 151--162, Pitman Res. Notes Math. Ser., 269, Longman Sci. Tech., Harlow, 1992. 

\bibitem{Kohn02}
\bysame, Superlogarithmic estimates on pseudoconvex domains and CR manifolds,  \emph{Ann. Math.}(2) \textbf{156} (2002), 213--248.

\bibitem{KohnNirenberg65}
Kohn, J.~J. and Nirenberg, L., \emph{Non-coercive boundary value
problems}, Comm. Pure Appl. Math. \textbf{18} (1965), 443--492.

\bibitem{McNeal02}
McNeal, Jeffery D., A sufficient condition for compactness of the
{$\overline\partial$}-{N}eumann operator, \emph{J. Funct. Anal.}
\textbf{195} (2002), Nr. 1, 190--205.

\bibitem{MunasingheStraube06}
Munasinghe, Samangi and Straube, Emil J., Complex tangential flows and compactness of the $\overline{\partial}$-Neumann operator, \emph{Pacific J. Math.} \textbf{232}, Nr.2 (2007), 343--354.

\bibitem{Nicoara06}
Nicoara, Andreea C., Global regularity for $\overline\partial_b$ on weakly pseudoconvex CR manifolds, \emph{Adv. Math.} \textbf{199}, no.2 (2006), 356--447.

\bibitem{Raich10}
Raich Andrew S., Compactness of the complex Green operator on CR-manifolds of hypersurface type, \emph{Math. Ann.} \textbf{348} (2010), 81--117.

\bibitem{RaichStraube08}
Raich, Andrew S. and Straube, Emil J., Compactness of the complex Green operator, \emph{Math. Res. Lett.} \textbf{15}, no.~4 (2008), 761--778.

\bibitem{ReedSimon80}
Reed, Michael and Simon, Barry, \emph{Methods of Modern Mathematical Physics, I: Functional Analysis}, revised and enlarged edition, Academic Press, 1980.

\bibitem{Sahutoglu09}
\c{S}ahuto\u{g}lu, S\"{o}nmez, Strong Stein neighborhood bases, preprint, 2009.

\bibitem{Shaw85}
Shaw, Mei-Chi, Global solvability and regularity for $\bar \partial$ on an annulus between two weakly pseudoconvex domains, \emph{Trans. Amer. Math. Soc.} \textbf{291}, no.~1  (1985), 255--267.

\bibitem{Shaw85b}
\bysame, $L\sp 2$-estimates and existence theorems for the tangential Cauchy-Riemann complex, \emph{Invent. Math.} \textbf{82}, no.1  (1985), 133--150.

\bibitem{Shaw10}
\bysame, The closed range property for $\overline{\partial}$ on domains with pseudoconcave boundary, In \emph{Complex Analysis: Several Complex Variables and Connections with PDEs and Geometry} (Fribourg 2008). Trends in Mathematics, Birkh\"{a}user Verlag, 2010.

\bibitem{Sibony87b}
Sibony, Nessim, Une classe de domaines pseudoconvexes, \emph{Duke
Math. J.} \textbf{55} , no. 2 (1987), 299--319.

\bibitem{Straube97}
Straube, Emil J., Plurisubharmonic functions and subellipticity of the $\overline{\partial}$-Neumann problem on non-smooth domains, \emph{Math. Research Letters} \textbf{4} (1997), 459--467.

\bibitem{Straube04}
\bysame, Geometric conditions which imply compactness of
the $\overline{\partial}$-Neumann operator, \emph{Ann. Inst. Fourier Grenoble} \textbf{54}, fasc. 3 (2004), 699--710.

\bibitem{Straube10}
\bysame, \emph{Lectures on the $\mathcal{L}^{2}$-Sobolev Theory of the $\overline{\partial}$-Neumann Problem}, ESI Lectures in Mathematics and Physics, European Mathematical Society, Z\"{u}rich, 2010.

\bibitem{Sweeney76}
Sweeney, W. J., A condition for subellipticity in Spencer's Neumann problem, \emph{J. Differential Equations} \textbf{21}, no.2 (1976), 316--362.

\bibitem{Trepreau86}
Tr\'{e}preau, J.-M. Sur le prolongement holomorphe des fonctions CR d\'{e}finies sur une hypersurface r\'{e}elle de classe $\mathbb{C}^2$ dans $\mathbb{C}^n$, \emph{Invent. Math.} \textbf{83}, no.3  (1986),583--592.

\bibitem{Zampieri08}
Zampieri, Giuseppe, \emph{Complex Analysis and CR-Geometry}, University Lecture Series 43, American Math. Soc., 2008.



\end{thebibliography}
\end{document}